\documentclass{amsart}
\usepackage{latexsym,amssymb}
\usepackage[english]{babel}
\usepackage{amsfonts, amsthm}
\usepackage{etoolbox}

\patchcmd{\quote}{\rightmargin}{\leftmargin 0.5 em \it \rightmargin}{}{}

\newtheorem{defi}{Definition}[section]
\newtheorem{prop}[defi]{Proposition}
\newtheorem{lem}[defi]{Lemma}
\newtheorem{rem}[defi]{Remark}
\newtheorem{ex}[defi]{Example}

\newtheorem{thm}[defi]{Theorem}

\newtheorem{conj}{Conjecture}
\theoremstyle{definition}
\newtheorem*{prob}{Open Problem}

\def\Z{\mathbb{Z}}

\def\eps{\varepsilon}

\begin{document}

\title{A problem on partial sums in abelian groups}

\author{S. Costa}
\address{DICATAM - Sez. Matematica, Universit\`a degli Studi di Brescia, Via
Branze 43, I-25123 Brescia, Italy}
\email{simone.costa@unibs.it}

\author{F. Morini}
\address{Dipartimento di Scienze Matematiche, Fisiche e Informatiche, Universit\`a di Parma,
Parco Area delle Scienze 53/A, I-43124 Parma, Italy}
\email{fiorenza.morini@unipr.it}

\author{A. Pasotti}
\address{DICATAM - Sez. Matematica, Universit\`a degli Studi di Brescia, Via
Branze 43, I-25123 Brescia, Italy}
\email{anita.pasotti@unibs.it}
\thanks{The results of this paper will be presented at HyGraDe 2017}
 
\author{M.A. Pellegrini}
\address{Dipartimento di Matematica e Fisica, Universit\`a Cattolica del Sacro Cuore, Via Musei 41,
I-25121 Brescia, Italy}
\email{marcoantonio.pellegrini@unicatt.it}

\begin{abstract}
In this paper we propose a conjecture concerning partial sums
of an arbitrary finite subset of an abelian group, that naturally arises 
investigating simple Heffter systems.
Then, we show its connection with related open problems and we present some results
about the validity of these conjectures.
\end{abstract}

\keywords{partial sum; Heffter system; cycle system}
\subjclass[2010]{05C25, 05C38}

\maketitle

\section{Introduction}

It is well known that difference methods  have a
primary role in the construction of combinatorial
designs of various kinds, see \cite{HB,BJL}. The continuous search
for more efficient ways to use these methods often
leads to intriguing problems which are very difficult
despite their easy statements.
Some
examples are the conjectures proposed by Alspach \cite{BH} and by Archdeacon et al. \cite{ADMS}. 

In order to describe these conjectures we introduce the concept of partial sums.
Let $A$ be a finite list of elements of a group $(G,+)$.
Let $(a_1,a_2,\ldots,a_k)$ be an ordering of the elements in $A$ and define the \emph{partial sums}
$s_1,s_2,\ldots,s_k$ by the formula $s_j=\sum_{i=1}^{j} a_i$ $(1\leq j \leq k)$.
An ordering of $A$ is said to be \emph{simple} if all the partial sums are distinct.

Several years ago Alspach made the following conjecture, whose validity would shorten some cases of known proofs 
about the existence of cycle decompositions.

\begin{conj}\label{Conj:als}
Let $A\subseteq \mathbb{Z}_v\setminus\{0\}$ such that
$\sum_{a\in A}a\neq0$. Then there exists an ordering of the elements of $A$ such that the partial sums are all
distinct and nonzero.
\end{conj}

The only published paper on this problem is by Bode and Harborth \cite{BH}. There, the authors proved that Conjecture \ref{Conj:als}
is valid if $|A|=v-1$ or $|A|=v-2$ or if $v\leq 16$ (the latter was obtained by computer verification).
Also they stated, without proof, that the conjecture is true if $|A|\leq5$.

In \cite{ADMS} Archdeacon et al. proposed a related conjecture.
\begin{conj}\label{Conj:ADMS}
Let $A\subseteq \mathbb{Z}_v\setminus\{0\}$.
Then there exists an ordering of the elements of $A$ such that the partial sums are all
distinct.
\end{conj}
In \cite{ADMS} the authors proved that Conjecture \ref{Conj:als} implies Conjecture \ref{Conj:ADMS}.
Then they proved that their conjecture is valid if $|A|\leq 6$ and made also a computer verification for $v\leq 25$.

Here, we propose the following conjecture:

\begin{conj}\label{Conj:nostra}
Let $(G,+)$ be an  abelian group. Let $A$ be a finite subset of $G\setminus\{0\}$ 
such that no $2$-subset $\{x,-x\}$ is contained in $A$ and 
with the property that $\sum_{a\in A} a=0$.
Then there exists an ordering of the elements of $A$ such that the partial sums are all
distinct.
\end{conj}

Clearly if $G=\Z_v$, Conjecture \ref{Conj:nostra} immediately follows from Conjecture \ref{Conj:ADMS}.
Indeed we believe that Conjecture \ref{Conj:ADMS} can be stated considering a subset of an abelian group and not necessarily of a cyclic one. 
About that, we proved by computer that Conjecture \ref{Conj:ADMS} is valid for any abelian group of order $\leq 23$.
Also, we have to point out that in \cite{ADMS}, proving their Theorems 3.1 and 3.2, the authors do not use
the hypothesis that $A$ is a subset of a cyclic group. In fact their proofs for $|A|\leq 6$ work, more in general, in an abelian group. 

Another motivation for extending Conjectures \ref{Conj:als} and \ref{Conj:ADMS} from cyclic groups to any (abelian) group is
because of their natural connection with the concept of sequenceable (or R-sequenceable) group, see \cite{E}.
In particular, Alspach et al. recently proved that any finite abelian group is either sequenceable or R-sequenceable, 
confirming the Friedlander-Gordon-Miller conjecture \cite{Al}.

We came across the problem of Conjecture \ref{Conj:nostra} studying Heffter systems, see \cite{H}, as explained in Section \ref{2}.
In Section \ref{3} we show how some known conjectures about graphs with prescribed edge-lengths can be stated in terms of partial sums
and we propose a related open problem.
In Section \ref{4} we prove the validity of Conjecture \ref{Conj:nostra} for subsets $A$ of size less than $10$;
we remark that we have also checked by computer the validity of our conjecture for abelian groups of order not exceeding $27$.
Finally, some observations about the above conjectures in the nonabelian case are presented in Section \ref{5}.

\section{Heffter systems and cyclic cycle systems}\label{2}

 Given an odd positive integer $v$,
an \emph{half-set} $A$ of $\Z_v$ is a subset of $\Z_v\setminus\{0\}$ of size $(v-1)/2$ such that 
no $2$-subset $\{x,-x\}$ is contained in $A$.
A \emph{Heffter system} $D(v,k)$
is a partition of an half-set of $\Z_v$ into parts of size $k$ such that the elements in each part sum to $0$.
Heffter himself introduced these systems to construct Steiner triple systems, see \cite{H}.
In \cite{A}, Archdeacon presented the related concept of a \emph{Heffter array}, which has various applications, see \cite{ABD,ADDY,CMPP,DM,DW}. 
A Heffter system is said to be \emph{simple} if each part admits a simple ordering.
Archdeacon in \cite{A} proved that a simple Heffter system $D(v,k)$ 
gives rise to a cyclic $k$-cycle system of order $v$, namely a decomposition of the complete graph $K_v$ of order $v$ into
cycles of length $k$ admitting $\Z_v$ as an automorphism group acting sharply transitively on the vertices.
We recall the following result.

\begin{prop}
A $k$-cycle system $\mathcal{C}$ of order $v$ is sharply vertex-transitive under $\Z_v$ if and only if, up to isomorphisms,
the following conditions hold:
\begin{itemize}
\item the set of vertices of $K_v$ is $\Z_v$;
\item for all $C=(c_1,c_2,\ldots,c_k) \in \mathcal{C}$, also $C+1:=(c_1+1,c_2+1,\ldots,c_k+1)\in \mathcal{C}$.  
\end{itemize} 
\end{prop}

Clearly, to describe a cyclic  $k$-cycle system of order $v$ it is sufficient to show a complete system $\mathcal{B}$
of representatives for the orbits of $\mathcal{C}$ under the action of $\Z_v$. The elements of $\mathcal{B}$
are called \emph{base cycles} of $\mathcal{C}$. 

The existence of (cyclic) cycle systems has been widely investigated, see \cite{BR,BDF}. 
In order to explain how to construct the cycles starting from a Heffter system
we have to introduce the concept of list of differences from a cycle and to show its usefulness 
in constructing cyclic cycle systems, see \cite{B}.

 \begin{defi}
Let  $C=(c_1,c_2,\ldots,c_{k})$ be a $k$-cycle with
vertices in an abelian group $G$.
The multiset
$$\Delta C = \{\pm(c_{h+1}-c_{h})\ |\ 1\leq h \leq k\},$$
where   the subscripts are taken modulo $k$,
is called the \emph{list of differences} from $C$.
\end{defi}

More generally, given a set $\mathcal{B}$ of $k$-cycles with vertices
in $G$, by $\Delta \mathcal{B}$ one means the union (counting
multiplicities) of all multisets $\Delta C$,
where  $C\in \mathcal{B}$.

\begin{thm}\label{thm:basecycles}
Let $\mathcal{B}$ be a set of $k$-cycles with vertices in $\Z_{v}$.
If $\Delta\mathcal{B}=\Z_v\setminus\{0\}$ then $\mathcal{B}$ is set of base cycles of
a cyclic $k$-cycle system of order $v$.
\end{thm}

Suppose now to have a simple Heffter system $D(v,k)$. For $1\leq i \leq \frac{v-1}{2k}$, let 
$\omega_i=(a_{i,1},a_{i,2},$ $\ldots,a_{i,k})$ 
be a simple ordering of the $i$-th part of $D(v,k)$. Let $C_i=(a_{i,1},a_{i,1}+a_{i,2},\ldots,\sum_{j=1}^{k-1}a_{i,j},\sum_{j=1}^{k}a_{i,j}=0)$.
Since the ordering $\omega_i$ is simple, then $C_i$ is a $k$-cycle of $K_v$; also it results $\Delta C_i=\pm\omega_i$.
Let $\mathcal{B}=\left\{C_1,\ldots,C_{\frac{v-1}{2k}}\right\}$, since $D(v,k)$ is a Heffter system then $\Delta\mathcal{B}=\Z_v\setminus\{0\}$. 
Hence, by Theorem \ref{thm:basecycles},  $\mathcal{B}$ is a set of base cycles of a cyclic $k$-cycle system of order $v$.

\begin{ex}
Consider the Heffter system $D(25,6)=\{\{3,1,4,-5,10,12\},\{2,7,$ $-9,6,8,11\}\}$.
Since the orderings $\omega_1=(1,3,4,-5,10,12)$ and  $\omega_2=(2,6,7,8,-9,$ $11)$ are simple, there
exists a cyclic $6$-cycle system of order $25$. The cycles associated to these orderings are
$C_1=(1,4,8,3,13,0)$ and $C_2=(2,8,15,23,14,0)$. Note that
 $\Delta C_1\cup \Delta C_2=\Z_{25}\setminus\{0\}$, so
 $\{C_1,C_2\}$ is a set of base cycles of a cyclic $6$-cycle system of order $25$, that is
$\mathcal{C}=\{C_1+i, C_2+i : i\in \Z_{25}\}$.
\end{ex}

Note that if, in the previous example, we replace the ordering $\omega_1$ with $\omega'_1=(1,4,-5,3,10,12)$,
we do not obtain a cycle, but the union of the cycles $(0,1,5)$ and $(0,3,13)$.
Hence, in order to obtain a system with cycles of the same length starting from a Heffter 
system $D$, it is necessary to require for the simplicity of $D$.
The validity of Conjecture \ref{Conj:nostra} would imply that any part of a Heffter system admits a simple ordering, namely that
any Heffter system is simple.

\section{Conjectures on graphs with prescribed edge-lengths}\label{3}

In this section we will see how the conjectures on partial sums of a given set 
presented in the Introduction are closely related to some conjectures on graphs with prescribed edge-lengths.
We recall that the {\it length} $\ell(x,y)$ of an edge $[x,y]$ of $K_v$ is so defined:
$$\ell(x,y)=min(|x-y|,v-|x-y|).$$
If $\Gamma$ is any subgraph of $K_v$, then the list of edge-lengths of $\Gamma$ is the multiset
$\ell(\Gamma)$ of the lengths (taken with their respective multiplicities) of  all the edges
of $\Gamma$.
For our convenience, if a list $L$ consists of
$m_1$ $a_1$'s, $m_2$ $a_2$'s, \ldots, $m_t$ $a_t$'s,
we will write $L=\left\{a_1^{m_1},a_2^{m_2},\ldots,a_t^{m_t}\right\}$, whose \emph{underlying set} is
 the set
$\{a_1,a_2,\ldots,a_t\}$.
Moreover, with an abuse of notation, by $\sum L$ we will mean $\sum_{a\in L} a$.

A famous conjecture about edge-lengths of a Hamiltonian path of the complete graph has been proposed by Buratti, Horak and Rosa, see \cite{HR}.

\begin{conj}\label{Conj:BHR}
Let $L$ be a list of $v-1$ positive integers not exceeding $\left\lfloor\frac{v}{2}\right\rfloor$.
Then there exists a Hamiltonian path $H$ of $K_v$ such that $\ell(H)=L$ if and only if 
for any divisor $d$ of $v$, the number of multiples of $d$ appearing in $L$ does not exceed $v-d$.
\end{conj}

It is not hard to see that, even if the 
statement in terms of edge-lengths of a Hamiltonian path is more elegant,  this conjecture can be formulated also in terms of partial sums of a given list.
In fact Conjecture \ref{Conj:BHR} can be also stated as:
\begin{quote}
Let $v$ be a positive integer and let $L$ be a list of $v-1$ nonzero elements of $\Z_v$. Then,
there exists a suitable sequence $(\eps_1,\ldots,\eps_{v-1})$, where each $\eps_i=\pm1$, and a suitable ordering 
$(a_1,\ldots,a_{v-1})$ of $L$ such that the partial sums of the sequence $(\eps_1a_1,\ldots,\eps_{v-1}a_{v-1})$
are exactly the elements of $\Z_v\setminus\{0\}$ if and only if for any divisor $d$ of $v$ the number of multiples of $d$ appearing in $L$
does not exceed $v-d$.
\end{quote}
For example, given the list $L=\{1^2,4^2,5\}$ in $\Z_6$, we can consider the sequence $(1,4,-1,4,-5)$ whose partial sums are
$1,5,4,2,3$. 

Meszka, Pasotti and Pellegrini proposed another related conjecture, see \cite{PPgc}.

\begin{conj}\label{MPP}
Let $v=2n+1$ be an odd integer and let $L$ be a list of $n$ positive integers not
exceeding $n$.
Then there exists a near $1$-factor $F$ of $K_v$ such that $\ell(F)=L$ if and only if
for any divisor $d$ of $v$ the number of multiples of $d$ appearing in $L$ does not
exceed $\frac{v-d}{2}$.
\end{conj}

\noindent Results about Conjectures \ref{Conj:BHR} and \ref{MPP} can be found in \cite{CDF, DJ, HR, PPdm, PPejc} and in \cite{PPgc,R}, respectively. 

Looking at Conjecture \ref{Conj:nostra} and at these conjectures it is quite natural to ask which properties have to satisfy
a list $L$ of $k$ elements of $\left\{1,2,\ldots,\left\lfloor\frac{v}{2}\right\rfloor\right\}$ in order to have a $k$-cycle $C$ of $K_v$ such that $\ell(C)=L$.

\begin{rem}\label{rem:trivial}
Let $v$ be a positive integer and let $L=\{a_1,\ldots,a_k\}$ be a list of $k$ elements
not exceeding $\left\lfloor\frac{v}{2}\right\rfloor$. A trivial necessary condition 
for the existence of a cycle $C$ of $K_v$ such that $\ell(C)=L$ is the existence of a list 
$(\eps_1,\ldots,\eps_{k})$, where each $\eps_i=\pm1$, such that $\sum_{i=1}^{k}\eps_i a_i\equiv0\pmod v$.
\end{rem}

For example, given the list $L=\{1^2,2,3,5^2\}$ we have $1-1-2+3+5+5\equiv 0\pmod{11}$
and $C=(0,5,10,9,1,2)$ is a cycle of $K_{11}$ such that $\ell(C)=L$.
It is easy to see that this condition it is not sufficient: in Proposition \ref{prop:nec}
we will show another necessary condition. Firstly, we state the following.

\begin{lem}
Let $L=\{a_1^{m_1},a_2^{m_2},\ldots,a_t^{m_t}\}$ be a list of $k$ elements of $\{1,2,\ldots,\left\lfloor\frac{v}{2}\right\rfloor\}$
such that $d=\gcd(v,a_1,a_2,\ldots,a_t)>1$.
There exists a cycle $C$ of $K_v$ such that $\ell(C)=L$ if and only if there exists
a cycle $C'$ of $K_{\frac{v}{d}}$ such that $\ell(C')=\left\{\left(\frac{a_1}{d}\right)^{m_1},\left(\frac{a_2}{d}\right)^{m_2},\ldots,\left(\frac{a_t}{d}\right)^{m_t}\right\}$.
\end{lem}

\begin{proof}
Let $C=(0,c_2,\ldots,c_k)$ be a $k$-cycle of $K_v$ such that $\ell(C)=L$. Note that since $d$ divides each element of $L$,
$d$ divides $c_i$ for any $i$. Moreover, since $C$ is a cycle, $c_i\not\equiv c_j\pmod v$ for any $i\neq j$,
so $k$ cannot exceed $\frac{v}{d}$, since every class of residues modulo $d$ intersects $\{0,\ldots,v-1\}$ in a set of that size.
From $c_i\not\equiv c_j\pmod v$ for any $i\neq j$, it immediately follows also that $\frac{c_i}{d}\not\equiv \frac{c_j}{d}\pmod{\frac{v}{d}}$ for any $i\neq j$, so
$C'=\left(0,\frac{c_2}{d},\ldots,\frac{c_k}{d}\right)$ is a $k$-cycle of $K_\frac{v}{d}$.
Obviously, $\ell(C')$ is the list of the elements of $L$ divided by $d$.
The converse can be done in a similar way. 
\end{proof}

Hence, without loss of generality, we can consider lists $\left\{a_1^{m_1},a_2^{m_2},\ldots,a_t^{m_t}\right\}$ of elements not exceeding $\left\lfloor\frac{v}{2}\right\rfloor$
such that $\gcd(v,a_1,a_2,\ldots,a_t)=1$. For example, instead of $v=20$ and $L=\{6^6,8^2\}$, we can consider $v'=10$ and
$L'=\{3^6,4^2\}$. Note that $C'=(0,4,8,5,2,9,6,3)$ is a cycle of $K_{10}$ such that $\ell(C')=L'$, so $C=(0,8,16,10,4,18,12,6)$ is a cycle of $K_{20}$ such that $\ell(C)=L$.

\begin{prop}\label{prop:nec}
Let $L=\left\{a_1^{m_1},a_2^{m_2},\ldots,a_t^{m_t}\right\}$ be a list of $k\leq v$ elements of $\{1,2,\ldots,\left\lfloor\frac{v}{2}\right\rfloor\}$
with $\gcd(v,a_1,a_2,\ldots,a_t)=1$.
If there exists a $k$-cycle $C$ of $K_v$ such that $\ell(C)=L$ then  for any divisor $d>1$ of $v$, 
 the number of multiples of $d$ appearing in $L$ does not exceed $\frac{k}{v}(v-d)$.
 \end{prop}
 
\begin{proof}
Let $C$ be a $k$-cycle of $K_v$ with $\ell(C)=L$ and let $d>1$ be a divisor of $v$.
Denote by $N$ the number of non-multiples of $d$ appearing in $L$.
By the hypothesis, $N\geq1$.
Consider the graph $\Gamma$
obtainable from $C$ by deleting all $N$ edges whose length is not divisible by $d$.
Then $\Gamma$ has exactly $N$ connected components 
some of which may be just isolated vertices.
It is also clear that all vertices of every connected component $K$ of $\Gamma$
are in the same residue class modulo $d$ so that $K$ has
at most $\frac{v}{d}$ vertices.
It follows that 
$k=|V(\Gamma)|\leq N\frac{v}{d}$, i.e. $N\geq \frac{k\cdot d}{v}$. Hence the number of multiples of $d$, namely $k-N$,
is at most $k-\frac{k\cdot d}{v}$, so we have the thesis.
\end{proof}

We have to point out that the necessary conditions of Remark \ref{rem:trivial} and Proposition \ref{prop:nec} are not sufficient. 
In fact if we take $v=8$ and we consider the list $L=\{3^4,4^4\}$, one can check that  a cycle $C$ of $K_{8}$ such that $\ell(C)=L$ does not exist.
Note that if $v$ is a prime, Proposition \ref{prop:nec} gives no necessary condition. Although, 
for example, there exists no cycle $C$ of $K_7$ such that $\ell(C)=\{1,2,3^5\}$. 
A special case is considered in the following remark.

\begin{rem}
Let $L=\{a^k\}$ with $1\leq a\leq \left\lfloor\frac{v}{2}\right\rfloor$. Then $(0,a,2a,\ldots,(k-1)a)$
is a cycle of $K_v$ if and only if the order of $a$ in $\Z_v$ is $k$, namely if $k=\frac{v}{\gcd(v,a)}$.
\end{rem}

So we propose the following.

\begin{prob}
Let $v$ be a positive integer and let $L$ be a list of positive integers not exceeding $\left\lfloor\frac{v}{2}\right\rfloor$. 
Find the necessary and sufficient conditions for $L$ in order to have a cycle $C$ of $K_v$
such that $\ell(C)=L$.
\end{prob}

\section{Proof of main results}\label{4}

In this section we prove that Conjecture \ref{Conj:nostra} holds for sets of small size.
Clearly, given an ordering $(a_1,a_2,\ldots,a_k)$ of a finite set $A\subseteq G\setminus\{0\}$,
for $1\leq i<j \leq k$, we have $s_i=s_j$ if and only if $\sum_{h=i+1}^{j}a_h=0$.
In the proof of the following theorems, we will show that an ordering $\omega$ of $A$ is simple, by checking that
there is no subsequence of consecutive elements of $\omega$, viewed as a $k$-cycle, which sums to $0$.
We recall that by $\sum A$ we mean $\sum_{a\in A} a$.
Clearly, if there exists $B \subsetneq A$ such that $\sum B=0$, then we may assume that $|B|\geq 3$,
as we are requiring that $A$ does not contain $2$-subsets of shape $\{x,-x\}$.
Furthermore, if $\sum B=0$ then $\sum (A\setminus B)=0$ (as $\sum A=0$) and so we may also assume $|B|\leq \frac{|A|}{2}$.

\begin{thm}\label{thm:8}
Conjecture \ref{Conj:nostra} is true for $|A|\leq 8$.
\end{thm}

\begin{proof}
Clearly, if there are no proper subsets $B$ of $A$ such that $\sum B=0$ we have the thesis.
So, suppose that there exists a subset $B \subsetneq A$ such that $\sum B=0$. As previously remarked, we may assume $3\leq |B|\leq \frac{|A|}{2}$.

If $|A|\leq 5$ the thesis immediately follows by above considerations.

Assume that $A=\{a_1,a_2,\ldots, a_6\}$ has order $6$ and, without loss of generality, we may suppose that $B=\{a_1,a_2,a_3\}$.
It is now clear that the ordering $(a_1,a_2,a_4,a_3,a_5,a_6)$ is simple, since $A\setminus B$ is the only other proper subset of $A$ that sums to $0$.

Now, assume that $A=\{a_1,a_2,\ldots,a_7\}$ is a set of size $7$ and let $T_1\subsetneq A$ such that $\sum T_1=0$. We can assume, without loss of generality, $T_1=\{a_1,a_2,a_3\}$.
If $T_1$ is the unique subset of $A$  of size $3$ such that $\sum T_1=0$, then the ordering $(a_1,a_2,a_4,a_3,a_5,a_6,a_7)$ is simple.
Suppose that there is another subset $T_2\subsetneq A$ of size $3$ such that $\sum T_2=0$.
It is easy to see that $|T_1\cap T_2|=1$. Without loss of generality, we may assume $T_2=\{a_3,a_4,a_5\}$. 
The ordering 
$$\omega=\left\{\begin{array}{ll}
(a_1,a_2,a_4,a_3,a_6,a_5,a_7) & \textrm { if }  a_1+a_5+a_7\neq 0,\\
(a_1,a_4,a_2,a_3,a_5,a_6,a_7) & \textrm { if }  a_1+a_5+a_7 = 0
\end{array} \right.$$
is simple.

Finally, we suppose that $A=\{a_1,a_2,\ldots,a_8\}$ is a set of size $8$. We split the proof into two cases.\\[2pt]
\underline{Case 1.} Assume that there exists a subset $Q\subsetneq A$ such that $|Q|=4$ and $\sum Q=0$.\\
First suppose that there exists a subset $T\subsetneq A$ such that $|T|=3$ and $\sum T=0$.
Clearly, $|Q\cap T|=1,2$ and replacing $Q$ with $A\setminus Q$ if $|Q\cap T|=1$, we may assume that
$Q=\{a_1,a_2,a_3,a_4\}$ and $T=\{a_3,a_4,a_5\}$. In this case the ordering
$$\omega=\left\{\begin{array}{ll}
(a_1,a_2,a_3,a_6,a_4,a_5,a_7,a_8) & \textrm { if } a_2+a_3+a_6\neq 0,\\
(a_1,a_2,a_3,a_7,a_4,a_5,a_6,a_8) & \textrm { if }  a_2+a_3+a_6 = 0
\end{array} \right.$$
is simple.
Hence,  we may assume that there are no subsets $T\subsetneq A$ of size $3$ such that $\sum T=0$.
In this case the ordering
$$\omega=\left\{\begin{array}{ll}
(a_1,a_2,a_3,a_5,a_4,a_6,a_7,a_8) & \textrm { if }  a_3+a_4+a_5+a_6\neq 0,\\
(a_1,a_2,a_3,a_5,a_4,a_7,a_6,a_8) & \textrm { if }  a_3+a_4+a_5+a_6 = 0
\end{array} \right.$$
is simple.\\[2pt]
\underline{Case 2.} Assume that there are no subsets $Q\subsetneq A$ such that $|Q|=4$ and $\sum Q=0$.\\
Hence, there is a subset $T_1\subsetneq A$ such that $|T_1|=3$ and $\sum T_1=0$: we may assume $T_1=\{a_1,a_2,a_3\}$. 
If there exists another subset $T_2\subsetneq A$ with $\sum T_2=0$, by our assumptions it must be that $|T_2|=3$. As in case $|A|=7$ we have $|T_1\cap T_2|=1$ and so we may assume $T_2=\{a_3,a_4,a_5\}$.
In this case the ordering
$$\omega=\left\{\begin{array}{ll}
(a_1,a_4,a_2,a_3,a_5,a_6,a_7,a_8) & \textrm { if } a_1+a_4+a_8\neq 0,\\
(a_1,a_4,a_2,a_3,a_5,a_6,a_8,a_7) & \textrm { if } a_1+a_4+a_8 = 0
\end{array} \right.$$
is simple.
Finally, suppose that $T_1$ is the unique subset of $A$ that sums to $0$.
In this case, the ordering $(a_1,a_2,a_4,a_3,a_5,a_6,a_7,a_8)$ is obviously simple.
\end{proof}

\begin{thm}\label{thm:9}
Conjecture \ref{Conj:nostra} is true for $|A|=9$.
\end{thm}

\begin{proof}
Clearly, if there are no proper subsets $B$ of $A$ such that $\sum B=0$ we have the thesis.
Hence in the following we will assume that there exists $B \subsetneq A$ such that $\sum B=0$.
As already remarked, we may assume   $3\leq |B|\leq 4$.\\[2pt] 
In the proof by $Q_i$ and $T_i$ we always mean a subset of $A$ of size $4$ and $3$, respectively.
Also by $Q_i$, $Q_j$ with $i\neq j$, we will mean $Q_i\neq Q_j$; analogously for $T_i$, $T_j$.
Let $Q_i, Q_j$ be such that $\sum Q_i=\sum Q_j=0$. Note that $|Q_i\cap Q_j|\neq 0$ otherwise $0\in A$,
also $|Q_i\cap Q_j|\neq3$ otherwise $Q_i=Q_j$. Hence $\sum Q_i=\sum Q_j=0$ implies
$|Q_i\cap Q_j|=1,2$.
Let $Q_i, T_j$ be such that $\sum Q_i=\sum T_j=0$. Note that $|Q_i\cap T_j|\neq 0$, otherwise $A$ contains a $2$-subset $\{x,-x\}$; also $|Q_i \cap T_j|\neq3$ otherwise $0 \in A$. Hence $\sum Q_i=\sum T_j=0$ implies $|Q_i\cap T_j|=1,2$.
Let $T_i, T_j$ be such that $\sum T_i=\sum T_j=0$. Note that $|T_i \cap T_j|\neq 2$, otherwise $T_i=T_j$,
hence $|T_i\cap T_j|=0,1$. Also if $\sum T_i=\sum T_j=0$ and $T_i\cap T_j=\emptyset$, then the triple $A\setminus(T_i \cup T_j)$ sums to $0$.\\[2pt]
We set $A=\left\{a_1,a_2,\ldots,a_9\right\}$. Firstly, we suppose that there is no $Q_i$ which sums to $0$.
Let $T_1=\left\{a_1,a_2,a_3\right\}$ such that $\sum T_1=0$.  If there is no other triple in $A$ which sum to $0$,
the ordering $(a_1,a_2,a_4,a_3,a_5,a_6,a_7,a_8,a_9)$
is simple. \\
\underline{Case a.} Assume there exists $T_2$ with $\sum T_2=0$ and $|T_1 \cap T_2|=0$. We can suppose $T_2=\left\{a_4,a_5,a_6\right\}$:
the ordering
$$\omega=\left\{\begin{array}{ll}
(a_1,a_2,a_4,a_3,a_5,a_7,a_6,a_8,a_9) & \textrm { if } a_3+a_5+a_7\neq 0,\\
(a_1,a_2,a_4,a_3,a_6,a_7,a_5,a_8,a_9) & \textrm { if } a_3+a_5+a_7 = 0
\end{array} \right.$$
is simple.\\[2pt]
\underline{Case b.} We now assume that for all $T_i\neq T_1$ which sums to $0$,  $|T_1 \cap T_i|=1$. Let $T_2=\left\{a_3,a_4,a_5\right\}$ with $\sum T_2=0$.
In this case the ordering
$$\omega=\left\{\begin{array}{ll}
(a_1,a_2,a_4,a_3,a_6,a_5,a_7,a_8,a_9) & \textrm { if } a_1+a_8+a_9\neq 0,\\
(a_1,a_2,a_4,a_3,a_6,a_5,a_8,a_7,a_9) & \textrm { if } a_1+a_8+a_9 = 0
\end{array} \right.$$
is simple.\\[2pt]
Suppose now that there exists $Q_i$ which sums to $0$. We split the proof into 3 cases.
\begin{itemize}
\item[1)] There exists only one $Q_1$ which sums to $0$.
\item[2)] There exist $Q_1$ and $Q_2$ which sum to $0$ such that $|Q_1\cap Q_2|=2$,
\item[3)] There exist at least two quadruples in $A$ which sum to $0$. For all such quadruples $Q_i\neq Q_j$, we have $|Q_i\cap Q_j|=1$.
\end{itemize}
\underline{Case 1.} We split this case into 3 subcases.
\begin{itemize}
\item[1.1)] There is no $T_1$ which sums to $0$.
\item[1.2)] There exists $T_1$ which sum to $0$ such that $|Q_1\cap T_1|=1$,
\item[1.3)] Otherwise.
\end{itemize}
\underline{Case 1.1.} We can suppose $Q_1=\{a_1,a_2,a_3,a_4\}$. Then $(a_1,a_2,a_3,a_5,a_4,a_6,a_7,a_8,$ $a_9)$ is a simple ordering of $A$.\\[2pt]
\underline{Case 1.2.} We can suppose $Q_1=\{a_1,a_2,a_3,a_4\}$ and $T_1=\{a_4,a_5,a_6\}$. By the hypothesis, $\sum Q_1=\sum T_1=0$.
Let $T_2=\{a_1,a_2,a_9\}$ and $T_3=\{a_1,a_8,a_9\}$. 
If $\sum T_2=0$ set $\omega=(a_1,a_2,a_5,a_4,a_7,a_6,a_8,a_9,a_3)$ 
and if $\sum T_3=0$ set $\omega=(a_1,a_2,a_3,a_5,a_4, a_8,a_6,a_7,a_9)$. In both cases, $\omega$ is a simple ordering of $A$.
So, assume that $\sum T_2,\sum T_3\neq 0$. In this case, 
$$\omega=\left\{\begin{array}{ll}
(a_1,a_2,a_3,a_5,a_4,a_7,a_6,a_8,a_9) & \textrm { if } a_2+a_3+a_5\neq 0,\\
(a_1,a_2,a_3,a_6,a_4,a_7,a_5,a_8,a_9) & \textrm { if }  a_2+a_3+a_5= 0
\end{array} \right.$$
is a simple ordering of $A$.\\[2pt]
\underline{Case 1.3.} We can suppose $Q_1=\{a_1,a_2,a_3,a_4\}$ and $T_1=\{a_3,a_4,a_5\}$. 
We note that
$$\omega=\left\{\begin{array}{ll}
(a_1,a_5,a_3,a_2,a_4,a_6,a_7,a_8,a_9) & \textrm { if } a_2+a_4+a_6\neq 0,\\
(a_1,a_5,a_3,a_2,a_4,a_7,a_6,a_8,a_9) & \textrm { if } a_2+a_4+a_6 = 0
\end{array} \right.$$
is a simple ordering of $A$.\\[2pt]
\underline{Case 2.} We can suppose $Q_1=\{a_1,a_2,a_3,a_4\}$ and $Q_2=\{a_3,a_4,a_5,a_6\}$. We recall that, by the hypothesis, $\sum Q_1=\sum Q_2=0$.
Let $T_1=\{a_2,a_4,a_7\}$, $T_2=\{a_3,a_4,a_7\}$ and $Q_3=\{a_1,a_6,a_8, a_9\}$.
If $\sum T_1=0$, set $\omega=(a_1,a_2,a_3,a_7,a_4,a_5,a_6,a_8, a_9)$; if $\sum T_2=0$ set 
$T_3=\{a_1,a_3,a_5\}$, $Q_4=\{a_4,a_6,a_7,a_8\}$  and 
$$\omega=\left\{\begin{array}{lll}
(a_1,a_3,a_5,a_4,a_7,a_6,a_8,a_9,a_2) & \textrm { if }\sum T_3\neq 0  \textrm{ and } \sum Q_4 \neq 0,\\
(a_1,a_3,a_5,a_4,a_7,a_6,a_9,a_8,a_2) & \textrm { if } \sum T_3\neq 0  \textrm{ and }\sum Q_4=0,\\
(a_1,a_2,a_3,a_5,a_4,a_7,a_6,a_8,a_9) & \textrm { if } \sum T_3=0;
\end{array} \right.$$
if $\sum Q_3=0$, set
$$\omega=\left\{\begin{array}{lll}
(a_1,a_3,a_7,a_4,a_5,a_6,a_8,a_9,a_2) & \textrm { if } a_1+a_3+a_7\neq 0,\\
(a_1,a_4,a_7,a_3,a_5,a_6,a_8,a_9,a_2) & \textrm { if } a_1+a_3+a_7=0.
\end{array} \right.$$
In all cases, $\omega$ is a simple ordering of $A$. 
So, assume $\sum T_1,\sum T_2,\sum Q_3\neq 0$. Set $T_4=\{a_3,a_5,a_7 \}$, $T_5=\{ a_1, a_6, a_9\}$; then 
$$\omega=\left\{\begin{array}{lll}
(a_1,a_2,a_4,a_7,a_3,a_5,a_6,a_8,a_9) & \textrm { if } \sum T_4 \neq 0,\\
(a_1,a_2,a_7,a_4,a_3,a_5,a_8,a_6,a_9) & \textrm { if }\sum T_4=0  \textrm{ and } \sum T_5 \neq 0,\\
(a_1,a_2,a_7,a_4,a_3,a_5,a_9,a_6,a_8) & \textrm { if } \sum T_4= \sum T_5 = 0.\\
\end{array} \right.$$
is a simple ordering of $A$.\\[2pt]
\noindent\underline{Case 3.} By the hypothesis, $\sum Q_1=\sum Q_2=0$ and $|Q_1\cap Q_2|=1$. Let $Q_1=\{a_1,a_2,a_3,a_4\}$ and 
$Q_2=\{a_4,a_5,a_6,a_7\}$.
Suppose firstly that does not exist $Q_3$ which sums to $0$.
Let $T_1=\{a_4,a_6,a_8\}$,  $T_2=\{a_3,a_4,a_5\}$, $T_3=\{a_2,a_3,a_5\}$, $T_4=\{a_1,a_3,a_5\}$ and $T_5=\{a_1,a_7,a_9\}$.

Assume $\sum T_1=0$. 
If either $\sum T_2=0$ or $\sum T_3=0$ take  
$\omega=(a_2,a_1,a_4, a_6,a_3, a_5,$ $a_8,a_7,a_9)$; if $\sum T_2,\sum T_3\neq 0$, take
$T_6=\{a_1,a_7,a_8\}$ and
$$\omega=\left\{\begin{array}{ll}
(a_2,a_1,a_4, a_5,a_3,a_7, a_9,a_6,a_8) & \textrm { if } \sum T_4=0,\\
(a_2,a_1,a_3, a_5,a_4,a_6, a_9,a_7,a_8) & \textrm { if } \sum T_4\neq 0  \textrm{ and } \sum T_6=0,\\
(a_1,a_2,a_3, a_5,a_4,a_6, a_9,a_7,a_8) & \textrm { if } \sum T_6\neq 0.
\end{array} \right.$$
In all these cases, $\omega$ is a simple ordering of $A$.

So, from now on, assume $\sum T_1\neq 0$.
If $\sum T_2=0$, the ordering
$$\omega=\left\{\begin{array}{ll}
(a_1,a_3,a_2, a_5,a_4,a_6, a_8,a_7,a_9) & \textrm { if } \sum T_5 \neq 0,\\
(a_2,a_3,a_1, a_5,a_4,a_6, a_8,a_7,a_9) & \textrm { if } \sum T_5=0
\end{array} \right.$$
is simple, so we may also suppose $\sum T_2\neq 0$.
To  conclude, it suffices to take $T_7=\{a_2,a_7,a_9\}$ and
$$\omega=\left\{\begin{array}{ll}
(a_1,a_2,a_7, a_3,a_5,a_4, a_6,a_8,a_9) & \textrm { if } \sum T_3 =\sum T_7=0,\\
(a_2,a_1,a_3, a_5,a_4,a_6, a_8,a_7,a_9) & \textrm { if  either } \sum T_3=0 \textrm{ and  } \sum T_7\neq 0 \\
 & \textrm{ or } \sum T_5=0 \textrm{ and }  \sum T_4\neq 0,\\
(a_3,a_2,a_1, a_5,a_4,a_6, a_8,a_7,a_9) & \textrm { if } \sum T_5=\sum T_4=0,\\
(a_1,a_2,a_3, a_5,a_4,a_6, a_8,a_7,a_9)& \textrm { if } \sum T_3, \sum T_5\neq 0.\\[2pt]
\end{array} \right.$$
Now, we have to suppose that there exists $Q_3$ which sums to $0$. Clearly, this implies that there is no $Q_4$ such that $\sum Q_4=0$ and $|Q_4\cap Q_i|=1$ for all $i=1,2,3$.
By the assumptions it follows that $|Q_1\cap Q_2|=|Q_1\cap Q_3|=|Q_2\cap Q_3|=1$. So
 we can suppose $Q_1=\{a_1,a_2,a_3,a_4\}$, $Q_2=\{a_4,a_5,a_6,a_7\}$ and $Q_3=\{a_1,a_7,a_8,a_9\}$.
We set $T_1=\{a_1,a_3,a_5\}$, $T_2=\{a_4,a_6,a_8\}$ and $T_3=\{a_2,a_7,a_9\}$.
If $\sum T_1, \sum T_2,\sum T_3 \neq0$ then $(a_2,a_1,a_3,a_5,a_4,a_6,a_8,a_7,a_9)$ works.
We can focus our attention on the case $\sum T_1=0$, namely the case  $\sum T_2=0$  ($\sum T_3=0$, respectively) can 
be done in a similar way replacing each $a_i$ with $a_{i+3}$ (with $a_{i+6}$, respectively) where the subscripts are 
taken modulo $9$. So in the following we assume $\sum T_1=0$ and set $T_4=\{a_4,a_5,a_8\}$.
If $\sum T_3=0$, take $\omega=(a_2,a_1,a_4, a_5,a_3,a_7, a_9,a_6,a_8)$
and if $\sum T_4=0$ take
$$\omega=\left\{\begin{array}{ll}
(a_3,a_1,a_6, a_5,a_4,a_2, a_9,a_7,a_8)& \textrm { if } a_3+a_7+a_8 \neq 0,\\
(a_2,a_1,a_3, a_6,a_4,a_5, a_9,a_7,a_8)& \textrm { if } a_3+a_7+a_8=0.
\end{array} \right.$$
In all cases $\omega$ is a simple ordering of $A$.
So, assume $\sum T_3,\sum T_4\neq 0$. In this case,
$(a_2,a_1,a_3,a_6,a_4,a_5,a_8,a_7,a_9)$ is a simple ordering of $A$.
\end{proof}

\section{Further developments}\label{5}

Looking at Conjectures \ref{Conj:als} and \ref{Conj:ADMS} presented in the Introduction, a natural question is to ask what happens if one considers finite subsets $A\subseteq G\setminus\{0\}$, where $G$ is any group, not necessarily cyclic.
For instance we verified, with the help of a computer, the validity of Conjecture \ref{Conj:als}, 
for \emph{all abelian groups} of order $|G|\leq 21$.
On the other hand, Conjecture \ref{Conj:als} cannot be generalized to nonabelian groups. For instance, consider the symmetric group $G=\mathrm{Sym}(3)$ and its subset $A=G\setminus\{0\}$ (we keep using the additive notation): any ordering of $A$ is such that  $\sum A\neq 0$, but there is no ordering of $A$ such that all the partial sums are distinct and nonzero.

We have already remarked in the Introduction what happens for Conjecture \ref{Conj:ADMS} if $G$ is an abelian group. 
Note that if we consider Conjectures \ref{Conj:als} and \ref{Conj:ADMS} in the case of abelian groups then, again,
Conjecture \ref{Conj:als} implies Conjecture \ref{Conj:ADMS}: in fact it suffices to apply the same proof of \cite[Proposition 1.1]{ADMS}.

For Conjecture \ref{Conj:ADMS} it is also natural to investigate the nonabelian case.
We made a computer verification of this conjecture for \emph{all groups} of order $|G|\leq 19$.
We have also the following theoretical result.

\begin{thm}\label{A5}
Let $(G,+)$ be a group and let $A\subseteq G\setminus\{0\}$ with $|A|\leq 5$.
Then there exists an ordering of the elements of $A$ such that the partial sums are all
distinct.
\end{thm}

\begin{proof}
If $|A|\leq 2$ it is obvious.
Suppose $|A|>2$ and let $p$ be the number of distinct $2$-subsets $\{x,-x\}$ contained in $A$.

Let $A=\{a_1,a_2,a_3\}$ be a subset of size $3$. If $p=0$, then the ordering $(a_1,a_2,a_3)$ is simple.
If $p=1$ we may assume $a_2=-a_1$ and take the ordering $(a_1,a_3,-a_1)$.

Let $A=\{a_1,a_2,a_3,a_4\}$ be a subset of size $4$.
If $p=0$, the ordering
$$\omega=\left\{\begin{array}{ll}
(a_1,a_2,a_3,a_4) &  \textrm{ if  } a_2+a_3+a_4\neq 0,\\
(a_2,a_1,a_3,a_4)  & \textrm{ if  } a_2+a_3+a_4=0
\end{array}\right.$$
is simple.
If $p=1$, we may assume $a_2=-a_1$: in this case, $(a_3,a_1,a_4,-a_1)$ is a simple ordering.
If $p=2$, we may assume $a_2=-a_1$ and $a_4=-a_3$: it is easy to see that $(a_1,a_3,-a_1,-a_3)$ is a simple ordering of $A$.

Now, let $A=\{a_1,a_2,a_3,a_4,a_5\}$ be a subset of size $5$.
First, we consider the case $p=0$. If $a_2+a_3+a_4+a_5\neq 0$, $a_3+a_4+a_5\neq 0$ and $a_2+a_3+a_4\neq 0$, the ordering $(a_1,a_2,a_3,a_4,a_5)$ is simple.
So, suppose that $a_2+a_3+a_4+a_5=0$. If $a_1+a_3+a_4\neq 0$, then the ordering $(a_2,a_1,a_3,a_4,a_5)$
is simple. So, assume also that $a_1+a_3+a_4=0$ and observe that
\begin{equation}\label{*}
\omega=\left\{\begin{array}{ll}
(a_2,a_1,a_3,a_5,a_4) &  \textrm{ if  } a_3+a_5+a_4\neq 0,\\
(a_5,a_3,a_2,a_4,a_1)  & \textrm{ if  } a_3+a_5+a_4=0
\end{array}\right.
\end{equation}
is a simple ordering of $A$.
Next, suppose $a_3+a_4+a_5=0$. The ordering $(a_5,a_1,a_3,a_4,$ $a_2)$ is simple, except when 
$a_1+a_3+a_4+a_2=0$. However, if this holds, in order to find a simple ordering of $A$ it suffices 
to reapply \eqref{*} to the set $\{a_1',a_2',a_3',a_4',a_5'\}$, where $a_1'=a_5$, $a_2'=a_1$, $a_3'=a_3$, $a_4'=a_4$ and $a_5'=a_2$.
Similarly, if $a_2+a_3+a_4=0$, we consider the set $\{a_1',a_2',a_3',a_4',a_5'\}$, where $a_1'=a_1$, $a_2'=a_5$, $a_3'=a_2$, $a_4'=a_3$, $a_5'=a_4$, and proceed as done previously.

If $p=1$, we may suppose that $a_2=-a_1$. If $\pm a_1+a_3+a_4\neq 0$, then $(a_5,a_1,a_3,a_4,-a_1)$ is a simple ordering of $A$. 
Suppose $\varepsilon \cdot a_1+a_3+a_4=0$ for some $\varepsilon=\pm 1$: in this case, the ordering
$$\omega=\left\{\begin{array}{ll}
(a_3,\varepsilon \cdot a_1, a_5,a_4, -\varepsilon \cdot a_1) &  \textrm{ if  } a_5+a_4\neq \varepsilon \cdot a_1,\\
(\varepsilon\cdot a_1, a_5,a_3,a_4,-\varepsilon\cdot a_1)&  \textrm{ if  } a_5+a_4= \varepsilon \cdot a_1
\end{array}\right.$$
is simple.

If $p=2$, then we may suppose that
$a_2=-a_1$ and $a_4=-a_3$. If $a_1+a_3-a_1-a_3\neq 0$, then the ordering $(a_5,a_1,a_3,-a_1,-a_3)$ is simple.
Assume $a_1+a_3-a_1-a_3=0$. Then the ordering 
$$\omega=\left\{\begin{array}{ll}
(-a_3,a_1,a_3,-a_1,a_5) &  \textrm{ if  } a_3-a_1+a_5\neq 0,\\
(-a_3,-a_1,a_3,a_1,a_5)  & \textrm{ if  } a_3-a_1+a_5=0
\end{array}\right.$$
is simple. In fact, since $a_1+a_3=a_3+a_1$, we have $\pm a_1+a_3\mp a_1+a_5=a_3+a_5\neq 0$.
\end{proof}

Alspach \cite{Als} recently proposed the following definition, closely related to Conjecture \ref{Conj:ADMS} in its generalized 
formulation: a finite group $G$ is said to be \emph{strongly sequenceable} if every Cayley digraph on $G$ admits either 
an orthogonal
directed path or an orthogonal directed cycle. 
When this path (respectively, cycle) has length $|G|-1$, we retrieve 
the concept of sequenceable (respectively, R-sequenceable)  group, see \cite{Al}.
The problem that Kalinowski and he propose is the classification of the strongly sequenceable groups. In this 
direction, Theorem \ref{A5} can be viewed as an intermediate step.

\section*{Acknowledgements}
The authors would like to thank Marco Buratti and Jeff Dinitz for useful discussion on this topic.

\end{document}